\documentclass[bibother]{asl}
\usepackage{rotate}

\usepackage[UKenglish]{babel}
\usepackage[utf8]{inputenc}
\usepackage{verbatim}

\usepackage{theoremref}
\usepackage{hyperref}

\usepackage{dsfont}
\usepackage{enumitem}

\usepackage{thmtools}

\usepackage{catoptions}
\makeatletter

\def\Autoref#1{%
	\begingroup
	\edef\reserved@a{\cpttrimspaces{#1}}%
	\ifcsndefTF{r@#1}{%
		\xaftercsname{\expandafter\testreftype\@fourthoffive}
		{r@\reserved@a}.\\{#1}%
	}{%
		\ref{#1}%
	}%
	\endgroup
}
\def\testreftype#1.#2\\#3{%
	\ifcsndefTF{#1autorefname}{%
		\def\reserved@a##1##2\@nil{%
			\uppercase{\def\ref@name{##1}}%
			\csn@edef{#1autorefname}{\ref@name##2}%
			\autoref{#3}%
		}%
		\reserved@a#1\@nil
	}{%
		\autoref{#3}%
	}%
}
\makeatother

\theoremstyle{plain}
\newtheorem{theorem}{Theorem}[section]
\newtheorem{corollary}[theorem]{Corollary}
\newtheorem{lemma}[theorem]{Lem\-ma}
\newtheorem{proposition}[theorem]{Prop\-o\-si\-tion}
\newtheorem{question}[theorem]{Question}
\newtheorem{conjecture}[theorem]{Conjecture}

\newtheorem{fact}[theorem]{Fact}
\newtheorem{observation}[theorem]{Observation}
\theoremstyle{definition}
\newtheorem{definition}[theorem]{Definition}

\newtheorem{remark}[theorem]{Remark}

\newtheorem*{shelahhasson}{Shelah--Hasson Conjecture}

\newcommand{\N}[0]{\mathbb{N}}
\newcommand{\Z}[0]{\mathbb{Z}}
\newcommand{\Q}[0]{\mathbb{Q}}
\newcommand{\R}[0]{\mathbb{R}}

\newcommand{\OO}[0]{\mathcal{O}}
\newcommand{\MM}[0]{\mathcal{M}}
\newcommand{\NN}[0]{\mathcal{N}}
\newcommand{\supp}[0]{\mathrm{supp}}

\newcommand{\pow}[1]{\!\left(\!\left( #1 \right)\!\right)}

\newcommand{\Lnew}[0]{\mathcal{L}}
\newcommand{\Lor}[0]{\mathcal{L}_{\mathrm{or}}}

\newcommand{\Lr}[0]{\mathcal{L}_{\mathrm{r}}}
\newcommand{\Lvf}[0]{\mathcal{L}_{\mathrm{vf}}}

\newcommand{\ol}[1]{\overline{#1}}

\newcommand{\vmin}[0]{v_{\min}}
\newcommand{\vnat}[0]{v_{\mathrm{nat}}}

\newcommand{\ac}[0]{\mathrm{ac}}

\newcommand{\Th}[0]{\mathrm{Th}}

\newcommand{\co}[0]{\colon\thinspace}

\newcommand{\thmref}[1]{Theorem~\ref{#1}}
\newcommand{\propref}[1]{Prop\-o\-si\-tion~\ref{#1}}
\newcommand{\lemref}[1]{Lemma~\ref{#1}}
\newcommand{\corref}[1]{Corollary~\ref{#1}}
\newcommand{\factref}[1]{Fact~\ref{#1}}
\newcommand{\conjref}[1]{Conjecture~\ref{#1}}
\newcommand{\defref}[1]{Definition~\ref{#1}}

\newcommand{\obsref}[1]{Observation~\ref{#1}}
\newcommand{\quref}[1]{Question~\ref{#1}}
\newcommand{\secref}[1]{Section~\ref{#1}}

\usepackage{xcolor}

\def\confname{}
\def\copyrightyear{2000}

\title{Strongly NIP almost real closed fields}

\thanks{    We started this research at the \emph{Model Theory, Combinatorics and Valued fields Trimester} at the Institut Henri Poincaré in March 2018. All three authors wish to thank the IHP for its hospitality.	
	We thank the anonymous referees of previous versions of this work for providing helpful comments.}

\author[L.~S.~Krapp]{Lothar Sebastian Krapp}
\revauthor{Krapp, Lothar Sebastian}
\author[S.~Kuhlmann]{Salma Kuhlmann}
\revauthor{Kuhlmann, Salma}
\author[G.~Lehéricy]{Gabriel Lehéricy}
\revauthor{Lehéricy, Gabriel}

\address{Fachbereich Mathematik und Statistik\\Universität Konstanz\\78457 Konstanz, Germany}
\email{sebastian.krapp@uni-konstanz.de}
\urladdr{http://www.math.uni-konstanz.de/\urltilde krapp/}
\thanks{The first author was supported by a doctoral scholarship of Studienstiftung des deutschen Volkes as well as of Carl-Zeiss-Stiftung, and by Werner und Erika Messmer-Stiftung.}

\address{Fachbereich Mathematik und Statistik\\Universität Konstanz\\78457 Konstanz, Germany}
\email{salma.kuhlmann@uni-konstanz.de}
\urladdr{https://www.mathematik.uni-konstanz.de/kuhlmann/}

\address{Fachbereich Mathematik und Statistik\\Universität Konstanz\\78457 Konstanz, Germany}
\curraddr{École supérieure d'ingénieurs Léonard-de-Vinci\\Pôle Universitaire Lé\-o\-nard de Vinci\\92 916 Paris La Défense Cedex, France}
\email{gabriel.lehericy@uni-konstanz.de}
\urladdr{http://www.math.uni-konstanz.de/\urltilde lehericy/}

\begin{document}

	\begin{abstract}
  The following conjecture is due to Shelah--Hasson: Any infinite strongly NIP field is either real closed, algebraically closed, or admits a non-trivial definable henselian valuation, in the language of rings. We specialise this conjecture to ordered fields in the language of ordered rings, which leads towards a systematic study of the class of strongly NIP almost real closed fields. As a result, we obtain a complete characterisation of this class.
	\end{abstract}

\maketitle

\section{Introduction}

The study of tame ordered algebraic structures has received a considerable amount of attention since the notion of o-minimality was introduced in \cite{pillay}. Frequently, the goal is to give a complete characterisation of such model theoretically well-behaved structures in terms of their algebraic properties. For instance, 
a (totally) ordered group is o-minimal if and only if it is abelian and divisible (cf.~\cite[Proposition~1.4, Theorem~2.1]{pillay}) and
an ordered field is o-minimal if and only if it is real closed  (cf.~\cite[Proposition~1.4, Theorem~2.3]{pillay}). In fact, these characterisations hold true under the more general tameness condition of weak o-minimality (cf.~\cite[page~117]{dickmann} and \cite[Theorem~5.1, Theorem~5.3]{macpherson}).
While o-minimality and weak o-minimality are comparatively strong tameness conditions, the property NIP (`not the independence property'), introduced in \cite{shelah3}, is at the other end of the spectrum; indeed, any o-minimal structure is also NIP (cf.~\cite[Corollary~3.10]{pillay} and \cite{pillay2}). 

A strategy to examine the class of NIP ordered algebraic structures is to first consider refinements of the property NIP. 
In this regard, we are mainly concerned with dp-minimal as well as strongly NIP\footnote{A strongly NIP theory is usually said to be strongly dependent.} ordered groups and fields.
Since any weakly o-minimal theory is dp-minimal (cf. \cite[Corollary~4.3]{dolich}), we obtain the following hierarchy:	
\begin{center}
	\textbf{o-minimal \!$\to$ weakly o-minimal\! $\to$ dp-minimal 
		\!$\to$ strongly NIP\! $\to$ NIP}
\end{center}
In particular, any divisible ordered abelian group and any real closed field are strongly NIP. A full algebraic characterisation of dp-minimal ordered fields follows from \cite[Theorem~6.2]{jahnke} (see \propref{prop:dpminarc}), but so far there has not been a systematic study of the strongly NIP ordered field case. 

With this paper, we contribute to the analysis of strongly NIP ordered fields, in light of a conjecture suggested by Shelah in \cite[Conjecture~5.34~(c)]{shelah}. This conjecture was verified by Johnson
for dp-minimal fields in \cite[Theorem 1.6]{johnson1} and more recently for dp-finite\footnote{Dp-finiteness can be classed between dp-minimality and strong NIP in the picture above.} fields in \cite[Theorem 1.2]{johnson2}.  
Shelah's conjecture was reformulated as follows in \cite[page 820]{dupont},    \cite[page~720]{halevi2}, \cite[page~2214]{halevi3} and \cite[page~183]{halevi}.

\begin{shelahhasson}
	Let $K$ be an infinite strongly NIP field. Then $K$ is either real closed, or algebraically closed, or admits a non-trivial $\Lnew_{\mathrm{r}}$-definable\footnote{Throughout this work `definable' always means `definable with parameters'.} henselian valuation.
\end{shelahhasson}

The Shelah--Hasson Conjecture specialised to ordered fields reads as follows.

\begin{conjecture}
	\label{conj:main}
	Let $(K,<)$ be a strongly NIP ordered field. Then $K$ is either real closed or admits a non-trivial $\Lor$-definable henselian valuation.
\end{conjecture}

The valuation and model theory of almost real closed fields (see \cite{delon} and also \defref{def:arc})  is well-understood. A main achievement of our paper is to show that \conjref{conj:main} is equivalent to the following.

\begin{conjecture}
	\label{conj:classification}
	Any strongly NIP ordered field $(K,<)$ is almost real closed.
\end{conjecture}

{  We highlight this result in the following theorem.\footnote{ This will be restated as \thmref{thm:main}.}
	\begin{theorem}
		\conjref{conj:main} and \conjref{conj:classification} are equivalent.
\end{theorem}}

Dp-minimal and more generally strongly NIP ordered abelian groups have already been fully classified (cf.~\cite[Prop\-o\-si\-tion~5.1]{jahnke} and \cite[Theorem~1]{halevi2}). Moreover, \conjref{conj:main} has already been verified for dp-minimal ordered fields in \cite[Corollary~6.6]{jahnke}.
By a careful analysis of the results of \cite{jahnke}, we deduce in \propref{prop:dpminarc} that also \conjref{conj:classification} holds for dp-minimal ordered fields. Actually, we prove that an ordered field is dp-minimal if and only if it is almost real closed \emph{with respect to some dp-minimal ordered abelian group $G$}. This latter re\-sult raises the question whether the analogous classification holds for strongly NIP ordered fields. We therefore address
the following question:

\begin{question}
	Is it true that an ordered field $(K,<)$ is strongly NIP if and only if it is almost real closed with respect to some strongly NIP ordered abelian group $G$? \label{query}
\end{question}

{ In \secref{sec:conjecturalclassification}, we prove our other main result:\footnote{ This will be restated as \thmref{thm:arcf}.}}

\begin{theorem}\label{thm:arcf2} 
	Let $(K,<)$ be an almost real closed field with respect to some ordered abelian group $G$. Then $(K,<)$ is strongly NIP if and only if $G$ is strongly NIP.
\end{theorem}

This answers positively the backward direction of \quref{query} about the classification of strongly NIP ordered fields. 
Thus, only the following question remains open.

\begin{question}\label{qu:forward}
	Is every strongly NIP ordered field $(K,<)$ almost real closed with respect to some strongly NIP ordered abelian group $G$?
\end{question}

Finally, we note that a positive answer to \quref{qu:forward} would verify \conjref{conj:classification}. Conversely, if \conjref{conj:classification} is verified, then by \thmref{thm:arcf2} the answer to \quref{qu:forward} is positive.

We conclude in \secref{sec:questions} by stating some further open questions motivated by this work.\footnote{A preliminary version of this work is contained in our arXiv preprint \cite{krapp}, which contains also a systematic study of $\Lor$-definable henselian valuations in ordered fields as well as of the class of ordered fields which are dense in their real closure. This systematic study, of independent interest, will be the subject of a separate publication { \cite{kkl2}}.}

\section{General Preliminaries}\label{sec:prelim}

The set of natural numbers with $0$ is denoted by $\N_0$, the set of natural numbers without $0$ by $\N$. Let $\Lnew_{\mathrm{r}} = \{+,-,\cdot,0,1\}$ be the language of rings, $\Lnew_{\mathrm{or}} = \Lnew_{\mathrm{r}} \cup \{<\}$  the language of ordered rings and $\Lnew_{\mathrm{og}} = \{+,0,<\}$ the language of ordered groups. Throughout this work, we abbreviate the $\Lnew_{\mathrm{r}}$-structure of a field $(K,+,-,\cdot,0,1)$ simply by $K$, the $\Lnew_{\mathrm{or}}$-structure of an ordered field $(K,+,-,\cdot,0,1, <)$ by $(K,<)$ and the $\Lnew_{\mathrm{og}}$-structure of an ordered group $(G,+,0, <)$ by $G$.

All notions on valued fields can be found in \cite{kuhlmann,engler}. 
Let $K$ be a field and $v$ a valuation on $K$. We denote the \textbf{valuation ring} of $v$ in $K$ by $\OO_v$, the \textbf{valuation ideal}, i.e.  the maximal ideal of $\OO_v$, by $\MM_v$, the \textbf{ordered value group} by $vK$ and the \textbf{residue field} $\OO_v/\MM_v$ by $Kv$. For $a \in \OO_v$ we also denote $a + \MM_v$ by $\ol{a}$. For an ordered field $(K,<)$ a valuation is called \textbf{convex} (in $(K,<)$) if the valuation ring $\OO_v$ is a convex subset of $K$. In this case, the relation $\ol{a} < \ol{b} : \Leftrightarrow \ol{a} \neq \ol{b} \wedge a < b$ defines an order relation on $Kv$ making it an ordered field. Note that in ordered fields, henselian valuations are always convex:

\begin{fact}{\rm \cite[Lemma~2.1]{knebusch}}\label{fact:hensconv}
	Let $(K,<)$ be an ordered field and let $v$ be a henselian valuation on $K$. Then $v$ is convex on $(K,<)$.
\end{fact}

Let $\Lvf = \Lr\cup \{\OO_v\}$ be the \textbf{language of valued fields}, where $\OO_v$ stands for a unary predicate. Let $(K,\OO_v)$ be a valued field. An atomic formula of the form $v(t_1) \geq v(t_2)$, where $t_1$ and $t_2$ are $\Lr$-terms, stands for the $\Lvf$-formula $t_1=t_2=0 \vee (t_2\neq 0 \wedge \OO_v(t_1/t_2))$. Thus, by abuse of notation, we also denote the $\Lvf$-structure $(K,\OO_v)$ by $(K,v)$. Similarly, we also call $(K,<,v)$ an ordered valued field. 
We say that a valuation $v$ is $\Lnew$-definable for some language $\Lnew\in \{\Lr,\Lor\}$ if its valuation ring is an $\Lnew$-definable subset of $K$.

For any ordered abelian groups $G_1$ and $G_2$, we denote the \textbf{lexicographic sum} of $G_1$ and $G_2$ by $G_1 \oplus G_2$. This is the abelian group $G_1 \times G_2$ with the lexicographic ordering $(a,b) < (c,d)$ if $a<c$, or $a=c$ and $b<d$.

Let $K$ be a field and let $v$ and $w$ be valuations on $K$. We write $v\leq w$ if and only if $\OO_v \supseteq \OO_w$. In this case we say that $w$ is \textbf{finer than $v$} and $v$ is \textbf{coarser than $v$}. Note that $\leq$ defines an order relation on the set of convex valuations of an ordered field. 
We call two elements $a,b \in K$ \textbf{archimedean equivalent} (in symbols $a \sim b$) if there is some $n\in \N$ such that $|a|<n|b|$ and $|b|<n|a|$. Let $G = \{[a] \mid a \in K^\times\}$ the set of archimedean equivalence classes of $K^\times$. Equipped with addition $[a]+[b] = [ab]$ and the ordering $[a] < [b]$ defined by $a \not\sim b \wedge |b| < |a|$, the set $G$ becomes an ordered abelian group. Then $ K^\times \to G, a \mapsto [a]$ defines a convex valuation on $K$. This is called the \textbf{natural valuation} on $K$ and denoted by $\vnat$.\footnote{Note that $\vnat$ is trivial if and only if $(K,<)$ is archimedean.}

Let $(k,<)$ be an ordered field and let $G$ be an ordered abelian group. We denote the \textbf{Hahn field} with coefficients in $k$ and exponents in $G$ by $k\pow{G}$. The underlying set of $k\pow{G}$ consists of all elements in the group product $\prod_{g\in G}k$ with well-ordered support, where the \textbf{support} of an element $s$ is given by $\supp s = \{g \in G \mid s(g) \neq 0\}$. We denote an element $s \in k\pow{G}$ by $s = \sum_{g \in G} s_gt^g$, where $s_g = s(g)$ and $t^g$ is the characteristic function on $G$ mapping $g$ to $1$ and everything else to $0$. The ordering on $k\pow{G}$ is given by $s > 0 : \Leftrightarrow s(\min \supp s) > 0$. Let $\vmin$ be the valuation on $k\pow{G}$ given by $\vmin(s) = \min \supp s$ for $s \neq 0$. Note that $\vmin$ is convex and henselian. Note further that if $k$ is archimedean, then $\vmin$ coincides with $\vnat$.

We repeatedly use the Ax--Kochen--Ershov Principle for ordered fields. This follows from \cite[Corollary 4.2 (iii)]{farre}, where all appearing levels in the premise  equal $1$ (cf.~\cite[page~916]{farre}).

\begin{fact}[Ax--Kochen--Ershov Principle]\label{fact:ake}
	Let $(K, <,v)$ and $(L,<,w)$ be two ordered henselian valued fields. Then $(Kv,<) \equiv (Lw,<)$ and $vK \equiv wL$ if and only if $(K,  <, v) \equiv (L,<, w)$.
\end{fact}

Since we do not use explicitly the definitions of the independence property (IP), `not the independence property' (NIP), strong NIP and dp-minimality, we refer the reader to  \cite{simon} for all definitions in this regard.
For a structure $\NN$, we say that $\NN$ is  NIP (respectively, strongly NIP and dp-minimal) if its complete theory $\Th(\NN)$ is NIP (respectively, strongly NIP and dp-minimal). A well-known example of an IP theory is the complete theory of the $\Lr$-structure $(\Z,+,-,\cdot,0,1)$ (cf. \cite[Example~2.4]{simon}). Since $\Z$ is parameter-free definable in the $\Lr$-structure $\Q$ (cf. \cite[Theorem~3.1]{robinson2}), also the complete $\Lr$-theory of $\Q$ has IP.
Any reduct of a strongly NIP structure is strongly NIP (cf. \cite[Claim~3.14 (3)]{shelah2}) and any reduct of a dp-minimal structure is dp-minimal (cf. \cite[Observation~3.7]{onshuus}).

\section{Almost Real Closed Fields}\label{sec:arc}

Algebraic and model theoretic properties of the class of almost real closed fields in the language $\Lr$ have been studied in \cite{delon}; in particular, \cite[Theorem~4.4]{delon} gives a complete characterisation of $\Lr$-definable henselian valuations. In the following, we prove some useful properties of almost real closed fields in the language $\Lor$.

\begin{definition} \label{def:arc}
	Let $(K,<)$ be an ordered field, $G$ an ordered abelian group and $v$ a henselian valuation on $K$. We call $K$ an \textbf{almost real closed field (with respect to $v$ and $G$)} if $Kv$ is real closed and $vK = G$. 
\end{definition}

Depending on the context, we may simply say that $(K,<)$ is an almost real closed field without specifying the henselian valuation $v$ or the ordered abelian group $G = vK$. 

\begin{remark}
	In \cite{delon}, almost real closed fields are defined as pure fields which admit a henselian valuation with real closed residue field. However, any such field admits an ordering, which is due to the Baer--Krull Representation Theorem (cf. \cite[page~37~f.]{engler}). We consider almost real closed fields as ordered fields with a fixed order.
\end{remark}

Due to \factref{fact:hensconv} and the following fact, we do not need to make a distinction between convex and henselian valuations in almost real closed fields.

\begin{fact}{\rm \cite[Prop\-o\-si\-tion~2.9]{delon}}\label{fact:convhens}
	Let $(K,<)$ be an almost real closed field. Then any convex valuation on $(K,<)$ is henselian.
\end{fact}

\cite[Prop\-o\-si\-tion~2.8]{delon} implies that the class of almost real closed fields in the language $\Lr$ is closed under elementary equivalence. We can easily deduce that this also holds in the language $\Lor$.

\begin{proposition}\label{prop:arcelem}
	Let $(K,<)$ be an almost real closed field and let $(L,  <) \equiv (K,<)$. Then $(L,<)$ is an almost real closed field.
\end{proposition}

\begin{proof}
	Since $L \equiv K$, we obtain by \cite[Prop\-o\-si\-tion~2.8]{delon} that $L$ admits a henselian valuation $v$ such that $Lv$ is real closed. Hence, $(L,<)$ is almost real closed.
\end{proof}

\begin{corollary}\label{cor:arcarc}
	Let $(K,<)$ be an ordered field. Then $(K,<)$ is almost real closed if and only if $(K,<) \equiv (\R\pow{G},<)$ for some ordered abelian group $G$. 
\end{corollary}

\begin{proof}
	The forward direction follows from \factref{fact:ake}. The backward direction is a consequence of \propref{prop:arcelem}.
\end{proof}

\begin{corollary}\label{cor:arcarc2}
	Let $(K,<)$ be an almost real closed field and let $G$ be an ordered abelian group. Then $(K\pow{G},<)$ is almost real closed.
\end{corollary}

\begin{proof}
	Let $v$ be a henselian valuation on $K$ such that $K$ is almost real closed with respect to $v$. Since $\vmin$ is henselian on $K\pow{G}$, we can compose the two henselian valuation $\vmin$ and $v$ in order to obtain a henselian valuation on $K\pow{G}$ with real closed residue field (cf.~\cite[Corollary~4.1.4]{engler}).
\end{proof}

\section{Strongly NIP Ordered Fields}\label{sec:conjecturalclassification}

In this section we study the class of strongly NIP ordered fields in light of \conjref{conj:main} and \conjref{conj:classification}.	
A special class of strongly NIP ordered fields are dp-minimal ordered fields. These are fully classified in \cite{jahnke}. In  \propref{prop:dparc} below we show that our query (see page~\pageref{query}) holds for dp-minimal ordered fields.
An ordered group $G$ is called \textbf{non-singular} if $G/pG$ is finite for all prime numbers $p$.

\begin{fact}\label{fact:dpgroup} {\rm \cite[Prop\-o\-si\-tion~5.1]{jahnke}}
	An ordered abelian group $G$ is dp-minimal if and only if it is non-singular.\footnote{The saturation condition in \cite{jahnke} can be dropped, as non-singularity of groups transfers via elementary equivalence.}
\end{fact}


\begin{fact}\label{fact:dpfield} {\rm \cite[Theorem~6.2]{jahnke}}
	An ordered field $(K,<)$ is dp-minimal if and only if there exists a non-singular ordered abelian group $G$ such that $(K, <) \equiv (\R\pow{G},<)$. 
\end{fact}

\begin{lemma}\label{lem:dpminresidue}
	Let $(K,<)$ be a dp-minimal almost real closed field with respect to some henselian valuation $v$. Then $vK$ is dp-minimal.
\end{lemma}

\begin{proof}
	Since $Kv$ is real closed, it is not separably closed. Thus, by \cite[Theorem~A]{jahnke2}, $v$ is definable in the Shelah expansion $(K,<)^{\mathrm{Sh}}$ (cf. \cite[Section~2]{jahnke2}) of $(K,<)$. By \cite[Observation~3.8]{onshuus}, also $(K,<)^{\mathrm{Sh}}$ is dp-minimal, whence the reduct $(K,v)$ is dp-minimal. 
	By \cite[Observation~1.4~(2)]{shelah}\footnote{We thank Yatir Halevi for pointing out this reference to us.}, any structure which is first-order interpretable in $(K,v)$ is dp-minimal (cf.\ also \cite{chernikov, jahnke}). Hence, also  $vK$ is dp-minimal.
\end{proof}

\begin{proposition}\label{prop:dparc}\label{prop:dpminarc}
	Let $(K,<)$ be an ordered field. Then $(K,<)$ is dp-minimal if and only if it is almost real closed with respect to a dp-minimal ordered abelian group.
\end{proposition}

\begin{proof}
	Suppose that $(K,<)$ is almost real closed with respect to a dp-minimal ordered abelian group $G$. By \factref{fact:dpgroup}, $G$ is non-singular. By \factref{fact:ake}, we have $(K,<) \equiv (\R\pow{G},<)$, which is dp-minimal by \factref{fact:dpfield}. Hence, $(K,<)$ is dp-minimal.
	
	Conversely, suppose that $(K,<)$ is dp-minimal. By \factref{fact:dpfield}, we have $(K,  <) \equiv (\R\pow{G},  <)$ for some non-singular ordered abelian group $G$. Since $(\R\pow{G},   <)$ is almost real closed, by \propref{prop:arcelem} also $(K,<)$ is almost real closed with respect to some henselian valuation $v$. By \lemref{lem:dpminresidue}, also $vK$ is dp-minimal, as required.
\end{proof}

As a result, we obtain a characterisation of dp-minimal archimedean ordered fields.

\begin{corollary}\label{cor:dparch}
	Let $(K,<)$ be a dp-minimal archimedean ordered field. Then $K$ is real closed.
\end{corollary}

\begin{proof}
	The only archimedean almost real closed fields are the archimedean real closed fields. This is due to the fact that any henselian valuation $w$ on an archimedean field $L$ is convex and thus trivial, whence the residue field of $Lw$ is equal to $L$. Thus, by \propref{prop:dpminarc}, any archimedean dp-minimal ordered field is real closed.	
\end{proof}

We now turn to strongly NIP almost real closed fields, aiming for a characterisation of these (see \thmref{thm:arcf}).
We have seen in \propref{prop:dparc} that every almost real closed field with respect to a dp-minimal ordered abelian group is dp-minimal. We obtain a similar result for almost real closed fields with respect to a strongly NIP ordered abelian group. The following two results will be exploited.


\begin{fact}\label{fact:snipgp}{\rm \cite[Theorem~1]{halevi2}}
	Let $G$ be an ordered abelian group. Then the following are equivalent:
	\begin{enumerate}[wide, label=(\arabic*), labelwidth=!, labelindent=6pt, itemsep=0pt, parsep=3pt, topsep=4pt]
		\item $G$ is strongly NIP.
		
		\item \label{fact:snipgp:2} $G$ is elementarily equivalent to a lexicographic sum of ordered abelian groups $\bigoplus_{i \in I} G_i$, where for every prime $p$, we have
		$|\{i \in I \mid pG_i \neq G_i \}| < \infty$,
		and for any $i \in I$, we have
		$|\{p \text{ prime} \mid [G_i:pG_i]=\infty \}| < \infty$.
	\end{enumerate}
\end{fact}

Details on angular component maps are given in \cite[Section~5.4~f.]{dries}. { Recall from \secref{sec:prelim} that any henselian valuation on an ordered field is convex (see \factref{fact:hensconv}) and thus naturally induces an ordering on the residue field given by $\ol{a} < \ol{b} : \Leftrightarrow \ol{a} \neq \ol{b} \wedge a < b$.}

\begin{observation}\label{obs:lift}
	Let $(K,<,v)$ be an ordered henselian valued field and let $\ac\colon K^\times\to Kv^\times$ be an angular component map. Suppose that the induced ordering of $K$ on $Kv$ is $\Lr$-definable. Then the ordering $<$ is definable in $(K,v,\ac)$.
\end{observation}

\begin{proof}
	Let $\varphi(x)$ be an $\Lr$-formula such that for any $a\in Kv$ we have
	$a\geq 0$ if and only if $Kv\models \varphi(a)$. Then the formula
	$x\neq 0\to  \varphi(\ac(x))$ defines the positive cone of the ordering $<$ on $K$.
\end{proof}

\begin{lemma}\label{lem:powstrongnip}
	Let $G$ be a strongly NIP ordered abelian group. Then the ordered Hahn field $(\R\pow{G}, <)$ is strongly NIP.
\end{lemma}

\begin{proof}
	If $K = \R\pow{G}$ is real closed, then we are done.
	Otherwise { let $v=\vmin$. Then} $(K,v)$ is $\ac$-valued with angular component map $\ac\co K \to \R$  given by $\ac(s) = s(v(s))$ for $s \neq 0$ and $\ac(0)=0$. Following a similar argument as \cite[page~188]{halevi}, we obtain that $(K,v,\ac)$ is a strongly NIP $\ac$-valued field; more precisely, 
	$(K,v,\ac)$ eliminates field quantifiers in the generalised Denef--Pas language (cf.~\cite[Section~5.6]{dries}, noting that both $K$ and $Kv$ have characteristic $0$)
	, whence by \cite[Fact~3.5]{halevi} we obtain that $(K,v,\ac)$ is strongly NIP. Since $\R$ is closed under square roots for positive elements, for any $a \in K$ we have $a \geq 0$ if and only if the following holds in $K$:
	$$\exists y \ y^2 = \ac(a)$$
	(see \obsref{obs:lift}).
	Hence, the order relation $<$ is definable in $(K,v,\ac)$. We obtain that $(K,<)$ is strongly NIP.
\end{proof}

\begin{proposition}\label{prop:classification2}
	Let $(K,<)$ be an almost real closed field with respect to a strongly NIP ordered abelian group and let $G$ be strongly NIP ordered abelian group. Then $(K\pow{G},<)$ is a strongly NIP ordered field.
\end{proposition}

\begin{proof}
	Let $H$ be a strongly NIP ordered abelian group such that $(K,<)$ is almost real closed with respect to $H$ and let $w$ be a henselian valution on $K$ with $wK=H$. 
	As in the proof of \corref{cor:arcarc2}, we can compose the valuation $\vmin$ on $K\pow{G}$ with $w$ on $K$ to obtain a henselian valuation on $K\pow{G}$ with real closed residue field and value group isomorphic to $G \oplus H$. Hence,
	$(K\pow{G},  <) \equiv (\R\pow{G \oplus H}, <)$. Since $G$ and $H$ are strongly NIP, also $G\oplus H$ is strongly NIP by \factref{fact:snipgp}. Hence, by \lemref{lem:powstrongnip}, also $(K\pow{G},<)$ is strongly NIP.
\end{proof}

\begin{corollary}\label{cor:classification22}
	Let $(K,<)$ be an almost real closed with respect to a hen\-se\-li\-an valuation $v$ such that $vK$ is strongly NIP. Then $(K,<)$ is strongly NIP.
\end{corollary}

\begin{proof}
	This follows immediately from \propref{prop:classification2} by setting $G = \{0\}$ and $H = vK$.
\end{proof}

For the proof of \thmref{thm:arcf}, we need one further result on general strongly NIP ordered fields, which will also be used for the proof of \thmref{thm:main}.

\begin{proposition}\label{prop:strongnipfieldresidue}
	Let $(K,<)$ be a strongly NIP ordered field and let $v$ be a henselian valuation on $K$. Then also $(Kv,<)$ and $vK$ are strongly NIP.
\end{proposition}

\begin{proof}
	Arguing as in the proof of \lemref{lem:dpminresidue}, we obtain that $v$ is definable in $(K, <)^{\mathrm{Sh}}$. Now $(K,<)^{\mathrm{Sh}}$ is also strongly NIP (cf. \cite[Observation~3.8]{onshuus}), whence $(K,  <,v)$ is strongly NIP. By \cite[Observation~1.4~(2)]{shelah}, both $(Kv,<)$ and $vK$ are strongly NIP, as they are first-order interpretable in $(K,  <,v)$.
\end{proof}

We obtain from \corref{cor:classification22} and \propref{prop:strongnipfieldresidue} the following characterisation of strongly NIP almost real closed fields.

\begin{theorem}\label{thm:arcf}
	Let $(K,<)$ be an almost real closed field with respect to some ordered abelian group $G$. Then $(K,<)$ is strongly NIP if and only if $G$ is strongly NIP.
\end{theorem}

\begin{remark}\label{rmk:arcf}
	\factref{fact:snipgp} and \thmref{thm:arcf} give us following complete characterisation of strongly NIP almost real closed fields: An almost real closed field $(K,<)$ is strongly NIP if and only if it is elementarily equivalent to some ordered Hahn field $(\R\pow{G},<)$ where $G$ is a lexicographic sum as in \factref{fact:snipgp}~\ref{fact:snipgp:2}.
\end{remark}

\section{Equivalence of Conjectures}\label{sec:conclusion}

Recall our two main conjectures.

\begingroup
\def\theconjecture{\ref{conj:main}}
\begin{conjecture}
	Let $(K,<)$ be a strongly NIP ordered field. Then $K$ is either real closed or admits a non-trivial $\Lor$-definable henselian valuation.
\end{conjecture}
\addtocounter{theorem}{-1}
\endgroup
\begingroup
\def\theconjecture{\ref{conj:classification}}
\begin{conjecture}
	Any strongly NIP ordered field is almost real closed.
\end{conjecture}
\addtocounter{theorem}{-1}
\endgroup

In this section, we show that \conjref{conj:main} and \conjref{conj:classification} are equivalent (see \thmref{thm:main}). 

\begin{remark}\label{rmk:mainrmk}
	\begin{enumerate}[wide, label=(\arabic*), labelwidth=!, labelindent=6pt, itemsep=0pt, parsep=3pt, topsep=4pt]
		
		\item\label{rmk:mainrmk:1} An ordered field is real closed if and only if it is o-minimal (cf.~\cite[Proposition~1.4, Theorem~2.3]{pillay}). Hence, for any real closed field $K$, if $\OO \subseteq K$ is a definable convex ring, its endpoints must lie in $K \cup \{\pm \infty\}$. This implies that any definable convex valuation ring must already contain $K$, i.e. is trivial. Thus, the two cases in the consequence of \conjref{conj:main} are exclusive.
		
		\item \label{rmk:mainrmk:2} Recall from \secref{sec:prelim} that the field $\Q$ is not NIP. By \cite[page~846]{ax}, the henselian valuation $\vmin$ is $\Lr$-definable in $\Q\pow{\Z}$. Hence, \propref{prop:strongnipfieldresidue} yields that $(\Q\pow{\Z}, <)$ is an example of an ordered field which is not real closed, admits a non-trivial $\Lor$-definable henselian valuation but is not strongly NIP.
		
	\end{enumerate}
\end{remark}

\lemref{lem:conj2impconj1} and \lemref{lem:conj1toconj2} below are used in the proof of \thmref{thm:main}.
For the first result, we adapt the proof of \cite[Lemma~3.7]{halevi} to the context of ordered fields.

\begin{lemma}\label{lem:conj2impconj1}
	Assume that any strongly NIP ordered field is either real closed or admits a non-trivial henselian valuation\footnote{Note that this valuation does not necessarily have to be $\Lor$-definable}. 	 Let $(K,<)$ be a strongly NIP ordered field. Then $(K,<)$ is almost real closed with respect to the canonical valuation, i.e. the finest henselian valuation on $K$.
\end{lemma}

\begin{proof}
	Let $(K,<)$ be a strongly NIP ordered field. If $K$ is real closed, then we can take the natural valuation. Otherwise, by assumption, the set of non-trivial henselian valuations on $K$ is non-empty. Let $v$ be the canonical valuation on $K$.
	By \propref{prop:strongnipfieldresidue}, we have that $(Kv,<)$ is strongly NIP. Note that $Kv$ cannot admit a non-trivial henselian valuation, as otherwise this would induce a non-trivial henselian valuation on $K$ finer than $v$. Hence, by assumption, $Kv$ must be real closed.
\end{proof}

The next result is obtained from an application of \cite[Prop\-o\-si\-tion~5.5]{halevi2}.

\begin{lemma}\label{lem:conj1toconj2}
	Let  $(K,<)$ be a strongly NIP ordered field which { is} not real closed but is almost real closed with respect to a henselian valuation $v$. Then there exists a non-trivial $\Lr$-definable henselian coarsening of $v$.
\end{lemma}

\begin{proof}
	By \propref{prop:strongnipfieldresidue}, we have that $vK=G$ is strongly NIP. Since $K$ is not real closed, $G$ is non-divisible (cf.\ \cite[~Theorem~4.3.7]{engler}). By \cite[Prop\-o\-si\-tion~5.5]{halevi2}, any henselian valuation with non-divisible value group on a strongly NIP field has a non-trivial $\Lr$-definable henselian coarsening. Hence, there is a non-trivial $\Lr$-definable henselian coarsening $u$ of $v$.
\end{proof}

\begin{theorem}\label{thm:main}
	\conjref{conj:main} and \conjref{conj:classification} are equivalent.
\end{theorem}

\begin{proof}
	Assume \conjref{conj:classification}, and let $(K,<)$ be a strongly NIP ordered field which is not real closed. 
	Then $(K,<)$ admits a non-trivial henselian valuation $v$. By \lemref{lem:conj1toconj2}, it also admits a non-trivial $\Lr$-definable henselian valuation.
	Now assume \conjref{conj:main}. Let $(K,<)$ be strongly NIP ordered field. By \lemref{lem:conj2impconj1}, we obtain that $K$ is almost real closed with respect to the canonical valuation $v$.
\end{proof}

As a final observation, we give two further equivalent formulations of \conjref{conj:classification} which follow from results throughout this work.

\begin{observation}
	The following are equivalent:
	
	\begin{enumerate}[wide, label=(\arabic*), labelwidth=!, labelindent=6pt, itemsep=0pt, parsep=3pt, topsep=4pt]
		
		\item\label{prop:finalobs:1} Any strongly NIP ordered field $(K,<)$ is almost real closed.
		
		\item\label{prop:finalobs:2} For any strongly NIP ordered field $(K,<)$, the natural valuation $\vnat$ on $K$ is henselian.
		
		\item\label{prop:finalobs:3} For any strongly NIP ordered valued field $(K,<,v)$, whenever $v$ is convex, it is already henselian.
		
	\end{enumerate}
\end{observation}

\begin{proof}
	\ref{prop:finalobs:1} implies \ref{prop:finalobs:3} by \factref{fact:convhens}. Suppose that \ref{prop:finalobs:3} holds and let $(K,<)$ be strongly NIP.  Now $\vnat$ is definable in the Shelah expansion $(K,<)^{\mathrm{Sh}}$, as it is the convex closure of $\Z$ in $K$. Hence, $(K,<,\vnat)$ is a strongly NIP ordered valued field. By assumption, $\vnat$ is henselian on $K$, which implies \ref{prop:finalobs:2}. Finally, suppose that \ref{prop:finalobs:2} holds. Let $(K,<)$ be a strongly NIP ordered field and $(K_1,<)$ an $\aleph_1$-saturated elementary extension of $(K,<)$. Then $K_1\vnat = \R$,  as any Dedekind cut on the rational numbers in $K_1$ is realised in $K_1$.
	
	{ More precisely, let $a\in \R$ and set $L=\{q\in \Q\mid q<a\}$ and $R=\{q\in \Q\mid a<q\}$. Then any finite subset of the $1$-type $p(x)=\{q<x\mid q\in L\}\cup \{x<q\mid q\in R\}$ is realised in $\Q$ and thus also in $K_1$. As $p(x)$ is countable, the $\aleph_1$-saturation of $K_1$ implies that $p(x)$ is realised in $K_1$ by some $\alpha\in K_1$. Since $K_1\vnat$ is archimedean, it embeds as an ordered field into $\R$, i.e.\ $(\Q,<)\subseteq (K_1\vnat,<)\subseteq (\R,<)$. Finally, by application of the residue map, for any $q_1,q_2\in \Q$ with $q_1<a<q_2$ we obtain $q_1\leq \overline{\alpha}\leq q_2$. Hence, $\overline{\alpha}=a$. Since $a$ was chosen arbitrary, we obtain $K_1\vnat =\R$.} 
	
	By assumption, $\vnat$ is henselian on $K_1$, whence $(K_1,<)$ is almost real closed. By \propref{prop:arcelem}, also $(K,<)$ is almost real closed.
\end{proof}

\section{Open Questions}\label{sec:questions}

We conclude with open questions connected to results throughout this work.
\conjref{conj:classification} for archimedean fields states that any strongly NIP ar\-chi\-me\-de\-an ordered field is real closed, as the only archimedean almost real closed fields are the real closed ones. \corref{cor:dparch} shows that any dp-minimal archimedean ordered fields is real closed. We can ask whether the same holds for all strongly NIP ordered fields.

\begin{question}\label{qu:sniprc}
	Let $(K,<)$ be a strongly NIP archimedean ordered field. Is $K$ necessarily real closed?
\end{question}

It is shown in \cite{kkl2} that any almost real closed field which is not real closed cannot be dense in its real closure. Thus,  any dp-minimal ordered field which is dense in its real closure is real closed.
Moreover, if \conjref{conj:classification} is true, then, in particular, a strongly NIP ordered field which is not real closed cannot be dense in its real closure. 

\begin{question}\label{qu:snipdenserc}
	Let $(K,<)$ be a strongly NIP ordered field which is dense in its real closure. Is $(K,<)$ real closed?
\end{question}

Note that \quref{qu:snipdenserc} is more general than \quref{qu:sniprc}, as a positive answer to \quref{qu:snipdenserc} would automatically tell us that any archimedean ordered field is real closed (since every archimedean field is dense in its real closure).

\end{document}